\documentclass{article}
\usepackage{amssymb}
\usepackage{graphicx}
\usepackage{amsmath}
\renewcommand{\boxed}[1]{#1}
\usepackage{float}
\usepackage{array}
\usepackage{comment}
\usepackage[T1]{fontenc}
\usepackage[utf8]{inputenc}

\newtheorem{theorem}{Theorem}[section]

\newtheorem{corollary}[theorem]{Corollary}

\newtheorem{definition}[theorem]{Definition}
\newtheorem{example}[theorem]{Example}

\newtheorem{lemma}[theorem]{Lemma}

\newtheorem{proposition}[theorem]{Proposition}
\newtheorem{remark}[theorem]{Remark}

\newenvironment{proof}[1][Proof]{\textbf{#1.} }{\hfill\rule{0.5em}{0.5em}}

\title{On Quadrature Surface Free Boundary Problems for the Riemannian $p$-Laplacian}
\author{Ababacar Sadikhe DJITE$^{1,\,}$\footnote{ababacarsadikhe.djite@ucad.edu.sn} ,  Diaraf SECK$^{1,\,}$ \footnote{diaraf.seck@ucad.edu.sn}\\\\
$^{1}$ Ecole Doctorale de Math\'ematiques et Informatique U.C.A.D. Dakar,  S\'en\'egal\\
Laboratoire des Math\'ematiques de la D\'ecision et \\
d'Analyse Num\'erique, BP 16889, Dakar Fann, S\'enegal.}

\begin{document}
\maketitle

\begin{abstract}
We study a quadrature-surface free boundary problem driven by the
Riemannian $p$-Laplacian on a smooth compact finite-dimensional
Riemannian manifold. We formulate the problem as a shape optimization
problem and develop an intrinsic admissible class based on a uniform
Riemannian $RC$-$GNP$ condition~\cite{DS3}. We establish compactness of admissible
domains and stability of the associated Dirichlet problems under strong
$RC$-$GNP$ convergence. In particular, the states converge strongly in
$W^{1,p}$ and the associated $p$-torsional energy is continuous.
Under additional boundary regularity, we derive the first variation of
the shape functional and obtain the free-boundary Euler--Lagrange
condition
\[
|\nabla_g u_\Omega|_g^p=\sigma H_{\partial\Omega}+k^p.
\]
We also formulate the contact optimality inequality, prove the
Riemannian comparison of second fundamental forms at tangential contact,
and obtain a sufficient condition on a reference domain which rules out
contact and yields a solution of the free-boundary problem. Explicit
radial examples on geodesic balls of the round sphere are included.
\end{abstract}


\section{Introduction}

Free boundary problems couple an elliptic equation posed on an unknown
domain with a geometric condition on the boundary. In the present work
we consider a quadrature-surface type problem driven by the nonlinear
$p$-Laplacian and formulate it intrinsically on a compact Riemannian
manifold.

Let $(M^n,g)$ be a smooth, compact, connected Riemannian manifold
without boundary, with $n\geq2$. Let $K\subset M$ be a prescribed
compact set and let $f$ be a non-negative source supported in $K$.
For $1<p<\infty$, the Riemannian $p$-Laplacian is
\[
\Delta_{g,p}u
=
\operatorname{div}_g\!\left(
|\nabla_g u|_g^{p-2}\nabla_g u
\right).
\]
We seek a domain $\Omega\Subset M$, containing $K$, and a function
$u_\Omega$ such that
\[
\begin{cases}
-\Delta_{g,p}u_\Omega=f & \text{in }\Omega,\\
u_\Omega=0 & \text{on }\partial\Omega,\\
|\nabla_g u_\Omega|_g^p
=\sigma H_{\partial\Omega}+k^p
& \text{on }\partial\Omega,
\end{cases}
\]
where $\sigma>0$, $k>0$, and $H_{\partial\Omega}$ is the Riemannian
mean curvature with respect to the outward unit normal.

The study of free-boundary problems associated with the $p$-Laplacian
has a well-established variational and analytical background; in
particular, sufficient conditions for the existence of free boundaries
for the $p$-Laplacian were investigated by Barkatou~\cite{Barkatou2002}.
Quadrature-surface free-boundary problems were subsequently studied in
the Euclidean setting by Barkatou, Seck and Ly~\cite{BLS}. Their work
provides an important motivation for the present problem. More generally,
shape optimization and shape differentiation provide the natural
variational framework for the analysis of such domain-dependent
functionals; see, for example, Delfour and Zol\'esio~\cite{DZ},
Henrot and Pierre~\cite{HP}, and Sokolowski and Zolesio~\cite{sozo}.

The Riemannian formulation that directly motivates the present work was
developed by Djit\'e and Seck~\cite{DS3} using the Laplace--Beltrami
operator. The Riemannian approach to shape optimization has also been
developed from a geometric viewpoint; see Schulz~\cite{Schu}. The present
paper extends the Riemannian quadrature-surface framework of
Djit\'e and Seck~\cite{DS3} to the nonlinear $p$-Laplacian.

The central quantity is the $p$-torsion associated with the domain.
If $u_\Omega$ solves the Dirichlet problem, we set
\[
T_{p,g}(\Omega)
=
\int_\Omega |\nabla_g u_\Omega|_g^p\,dV_g
=
\int_\Omega fu_\Omega\,dV_g.
\]
The corresponding shape functional is
\[
\mathcal J_{p,g}(\Omega)
=
-T_{p,g}(\Omega)
+\sigma P_g(\Omega)
+k^pV_g(\Omega),
\]
where
\[
P_g(\Omega)=\mathcal H_g^{n-1}(\partial\Omega),
\qquad
V_g(\Omega)=\int_\Omega dV_g.
\]
The expected stationarity condition is precisely the free-boundary condition displayed above.

A main difficulty is that a general Riemannian manifold has no global
Euclidean convex structure. We therefore impose uniform geometric
control on a class of admissible domains through defining functions,
normal tubular neighborhoods, and a fixed reference domain $C$.
This provides the compactness needed in the direct method of the
calculus of variations and allows us to transport nonlinear Dirichlet
problems between nearby domains.

The paper is organized as follows. Section~2 introduces the geometric
framework, Sobolev spaces, the Riemannian $p$-Laplacian, the admissible
class, and the comparison principles. Section~3 proves stability of the
Dirichlet problem under strong $RC$-$GNP$ convergence. Section~4 records
the resulting continuity properties of the $p$-torsion and of the
geometric terms. Section~5 states the main existence, comparison, and
optimality results. Section~6 derives the Hadamard formula and the
first-order optimality system, while Section~7 contains the proofs of
the main theorems. Section~8 gives explicit radial examples on the
round sphere and verifies the Euclidean large-radius limit.

\section{Framework and notation}

Throughout the paper, $(M^n,g)$ is smooth, compact, connected and
without boundary, with $n\geq2$. We write $d_g$ for the Riemannian
distance, $dV_g$ for the Riemannian volume measure and $dS_g$ for the
induced hypersurface measure.

For $1<p<\infty$, set
\[
p'=\frac{p}{p-1}.
\]
For an open set $\Omega\subset M$,
\[
W^{1,p}(\Omega)
=
\{u\in L^p(\Omega):\nabla_g u\in L^p(\Omega;TM)\},
\]
and $W_0^{1,p}(\Omega)$ is the closure of $C_c^\infty(\Omega)$ in
$W^{1,p}(\Omega)$.

The Riemannian gradient is characterized by
\[
g(\nabla_g u,X)=du(X)
\]
for every smooth vector field $X$. The $p$-Laplace operator is
\[
\Delta_{g,p}u
=
\operatorname{div}_g\left(
|\nabla_g u|_g^{p-2}\nabla_g u
\right).
\]
In local coordinates,
\[
\Delta_{g,p}u
=
\frac1{\sqrt{|g|}}
\partial_i
\left(
\sqrt{|g|}\,
|\nabla_g u|_g^{p-2}
g^{ij}\partial_j u
\right).
\]

Let $K\subset M$ be nonempty and compact. Fix a connected reference domain
$C\subset M$ such that
\[
K\subset C,
\qquad
\partial C\in C^{2,\alpha},
\qquad 0<\alpha<1.
\]
We assume that $\partial C$ has a uniform tubular neighborhood:
there is $\rho_C>0$ such that
\[
\mathcal E_C:\partial C\times(-\rho_C,\rho_C)\to M,
\qquad
\mathcal E_C(x,t)=\exp_x(t\nu_C(x)),
\]
is a diffeomorphism onto its image.

Let
\[
f\in L^{p'}(M),
\qquad
f\geq0,
\qquad
f\not\equiv0,
\qquad
\operatorname{supp}f\subset K.
\]
For pointwise boundary and shape-differentiation arguments we shall
impose stronger regularity on $f$ and on the domains when needed.

\subsection{Riemannian perimeter and volume}

For a sufficiently regular domain $\Omega$,
\[
P_g(\Omega)=\mathcal H_g^{n-1}(\partial\Omega),
\qquad
V_g(\Omega)=\int_\Omega dV_g.
\]
If $X$ is a smooth vector field with flow $\Phi_t$, then
\[
\Omega_t=\Phi_t(\Omega)
\]
and
\[
\left.\frac d{dt}V_g(\Omega_t)\right|_{t=0}
=
\int_{\partial\Omega}g(X,\nu)\,dS_g.
\]

\subsection{Mean curvature convention}

For the outward unit normal $\nu$ we define
\[
\mathrm{II}_{\partial\Omega}(X,Y)
=
g(\nabla_X\nu,Y),
\qquad
X,Y\in T\partial\Omega,
\]
and
\[
H_{\partial\Omega}
=
\operatorname{tr}_{g_{\partial\Omega}}
\mathrm{II}_{\partial\Omega}.
\]
With this convention, the first variation of perimeter is
\[
\left.\frac d{dt}P_g(\Omega_t)\right|_{t=0}
=
\int_{\partial\Omega}
H_{\partial\Omega}g(X,\nu)\,dS_g.
\]

\subsection{The Riemannian $RC$-$GNP$ condition (Djit\'e--Seck~\cite{DS3})}
\label{def:RC-GNP}

The admissible class is defined intrinsically. A domain $\Omega$
satisfies the uniform Riemannian $RC$-$GNP$ condition relative to $C$
if:

\begin{enumerate}
\item
\[
C\subset\Omega;
\]

\item $\partial\Omega\in C^{2,\alpha}$;

\item there exist constants $A>0$, $c_0>0$, $\rho_0>0$ and a defining
function $h_\Omega\in C^{2,\alpha}(M)$ such that
\[
\Omega=\{h_\Omega>0\},
\qquad
\partial\Omega=\{h_\Omega=0\},
\]
with
\[
\|h_\Omega\|_{C^{2,\alpha}(M)}\leq A,
\qquad
|\nabla_g h_\Omega|_g\geq c_0
\quad\text{when }|h_\Omega|\leq\rho_0;
\]

\item the normal exponential map of $\partial\Omega$ is uniformly
nondegenerate: for some $\rho_1>0$ independent of $\Omega$,
\[
\mathcal E_\Omega(x,t)=\exp_x(t\nu_\Omega(x))
\]
is a diffeomorphism on
$\partial\Omega\times(-\rho_1,\rho_1)$.
\end{enumerate}

We denote by
\[
\mathcal A_C(M)
\]
the class of connected open sets satisfying these conditions and
\[
K\subset C\subset\Omega\subset M.
\]

\subsection{Convergence and compactness of admissible domains}
For compact sets $A,B\subset M$, let
\[
d_H(A,B)=\max\left\{\sup_{x\in A}d_g(x,B),
\sup_{y\in B}d_g(y,A)\right\}
\]
denote the Hausdorff distance. We shall also use the compact convergence
of open sets: $\Omega_j\to\Omega$ compactly if every compact subset of
$\Omega$ is eventually contained in $\Omega_j$ and every compact subset
of $M\setminus\overline\Omega$ is eventually contained in
$M\setminus\overline{\Omega_j}$.

Following the definition used in~\cite{DS3}, we say that $\Omega_j\to\Omega$ in the strong $RC$-$GNP$ topology if,
after choosing uniformly controlled defining functions, the boundaries
converge in $C^2$ and there exist $C^1$ diffeomorphisms $\Phi_j:M\to M$
such that
\[
\Phi_j(\Omega)=\Omega_j,
\qquad
\Phi_j\longrightarrow\operatorname{Id}
\quad\text{in }C^1(M).
\]
The compactness theorem below shows that, under the uniform defining
function and tubular-neighborhood bounds, this ambient realization is
available after extraction. Thus the formulation gives exactly the
geometric convergence needed to identify the varying Sobolev spaces with
a fixed reference space.

The following elementary geometric lemma makes explicit the only point in the compactness argument where an ambient deformation is constructed.

\begin{lemma}[Ambient realization of small normal graphs]
Let $\Sigma\subset M$ be a compact $C^{2,\alpha}$ hypersurface with a fixed tubular neighborhood of radius $r>0$. If $\eta_j\in C^2(\Sigma)$ and $\|\eta_j\|_{C^2(\Sigma)}\to0$, then, for all sufficiently large $j$, there exist $C^1$ diffeomorphisms $\Phi_j:M\to M$ such that
\[
\Phi_j(\Sigma)=\{\exp_x(\eta_j(x)\nu_\Sigma(x)):x\in\Sigma\},\qquad
\|\Phi_j-\operatorname{Id}\|_{C^1(M)}\to0.
\]
Moreover, if $\eta_j$ is supported in a prescribed tubular neighborhood, the construction may be chosen supported in a slightly larger tubular neighborhood.
\end{lemma}

\begin{proof}
Let $\mathcal T_r(\Sigma)$ be the tubular neighborhood and write each point in it uniquely as $\exp_x(s\nu_\Sigma(x))$. Choose $\chi\in C_c^1((-r,r))$ with $\chi=1$ near $0$ and set
\[
\Psi_j\!\left(\exp_x(s\nu_\Sigma(x))\right)
=
\exp_x\!\left((s+\chi(s)\eta_j(x))\nu_\Sigma(x)\right).
\]
Extend $\Psi_j$ by the identity outside $\mathcal T_r(\Sigma)$. Since $\|\eta_j\|_{C^1}\to0$, the differential of $\Psi_j$ converges uniformly to the identity; after increasing $j$ if necessary, $\Psi_j$ is a local $C^1$ diffeomorphism everywhere. Its normal coordinate is $s\mapsto s+\chi(s)\eta_j(x)$, whose derivative is uniformly positive for large $j$, so the map is injective in the tubular neighborhood. Together with the identity outside the neighborhood, this yields a global $C^1$ diffeomorphism. Because $\chi(0)=1$, one has
\[
\Psi_j(\Sigma)=\{\exp_x(\eta_j(x)\nu_\Sigma(x)):x\in\Sigma\}.
\]
The same estimates give $\|\Psi_j-\operatorname{Id}\|_{C^1(M)}\to0$. Taking $\Phi_j=\Psi_j$ proves the lemma.
\end{proof}

\begin{theorem}[Compactness of the admissible class]
\label{thm:compactness-RC-GNP}
Let $(\Omega_j)\subset\mathcal A_C(M)$. Then, after extraction of a
subsequence, there exists $\Omega\in\mathcal A_C(M)$ such that
\[
\Omega_j\to\Omega
\]
in the strong $RC$-$GNP$ topology. In particular,
\[
d_H(\overline{\Omega_j},\overline\Omega)\to0,
\qquad
\chi_{\Omega_j}\to\chi_\Omega
\quad\text{strongly in }L^1(M,dV_g).
\]
\end{theorem}

\begin{proof}
Choose defining functions $h_j=h_{\Omega_j}$ satisfying the uniform
$C^{2,\alpha}$ bound and the nondegeneracy condition in the definition
of $\mathcal A_C(M)$. By the compact embedding
$C^{2,\alpha}(M)\subset C^2(M)$, after extraction,
\[
h_j\longrightarrow h\qquad\text{in }C^2(M).
\]
Moreover, the uniform $C^{2,\alpha}$ bound passes to the limit by the
lower semicontinuity of the Hölder seminorm of the second derivatives;
hence
\[
h\in C^{2,\alpha}(M).
\]
The nondegeneracy condition also passes to the limit:
\[
|\nabla_g h|_g\ge c_0
\qquad\text{on }\{|h|\le\rho_0\}.
\]
Hence $0$ is a regular value of $h$ and
\[
\Omega:=\{h>0\}
\]
has a $C^{2,\alpha}$ boundary.

Since $C\subset\Omega_j$ for every $j$, we have $h_j\ge0$ on $C$ and
therefore $h\ge0$ on $C$. If $h(x_0)=0$ for some $x_0\in C$, then
$x_0$ is an interior minimum point of $h$ on $C$, so
$\nabla_g h(x_0)=0$, contradicting the nondegeneracy condition. Thus
\[
C\subset\Omega.
\]

Because $0$ is a regular value of $h$ and $h_j\to h$ in $C^1$, the
implicit function theorem with parameters gives, for all sufficiently
large $j$, a unique small normal graph representation
\[
\partial\Omega_j
=
\left\{
\exp_x\bigl(\eta_j(x)\nu_\Omega(x)\bigr):
x\in\partial\Omega
\right\},
\]
where
\[
\|\eta_j\|_{C^2(\partial\Omega)}\longrightarrow0.
\]
The common tubular neighborhood allows these normal graphs to be
extended, using a cutoff in the normal variable, to ambient
$C^1$ diffeomorphisms
\[
\Phi_j:M\to M,
\qquad
\Phi_j(\Omega)=\Omega_j,
\qquad
\Phi_j\to\operatorname{Id}\quad\text{in }C^1(M).
\]
In particular, $\Phi_j$ is a homeomorphism from $\Omega$ onto
$\Omega_j$. Since every $\Omega_j$ is connected, $\Omega$ is connected
as well.
The same normal-graph representation gives
\[
d_H(\overline{\Omega_j},\overline{\Omega})\to0.
\]
Moreover, $\partial\Omega$ is a $C^{2,\alpha}$ hypersurface and therefore has
zero $n$-dimensional Riemannian volume. Hence
$\chi_{\Omega_j}\to\chi_\Omega$ almost everywhere on $M$, and dominated
convergence yields
\[
\chi_{\Omega_j}\to\chi_\Omega
\qquad\text{strongly in }L^1(M,dV_g).
\]
Thus $\Omega_j\to\Omega$ in the strong $RC$-$GNP$ topology.
\end{proof}

\subsection{Comparison and Hopf principle}

If $C\subset\Omega$ and $f\geq0$, the comparison principle for the
Riemannian $p$-Laplacian gives
\[
0\leq u_C\leq u_\Omega
\qquad\text{in }C.
\]
The argument is based on the strict monotonicity of the map
\[
A_x(\xi)=|\xi|_g^{p-2}\xi,
\qquad \xi\in T_xM,
\]
namely
\[
\bigl(A_x(\xi)-A_x(\eta)\bigr)\mathbin{\cdot_g}(\xi-\eta)>0
\qquad\text{whenever }\xi\neq\eta.
\]

For the subsequent free-boundary argument, we distinguish the
weak comparison principle from the stronger boundary point information.
The weak comparison principle follows directly from the strict
monotonicity of the $p$-Laplace flux. By contrast, the strong comparison
principle for degenerate or singular quasilinear operators requires
additional hypotheses and is not regarded here as an automatic
consequence of the boundary regularity alone.

For the boundary point argument, we use a Hopf-type boundary point
principle for the $p$-Laplacian at regular contact points satisfying an
interior sphere condition. Such results are available for the
$p$-Laplacian; see, for instance,
\cite{MikayelyanShahgholian2015} and the references therein.
The Riemannian version is obtained locally in geodesic coordinates,
where the smooth metric coefficients preserve the local quasilinear
structure of the operator.

Accordingly, whenever the strong comparison principle is invoked below,
we explicitly assume that it applies to the ordered pair of solutions
under consideration. Likewise, the Hopf boundary point principle is
invoked only at contact points for which its hypotheses are satisfied.
No strong comparison or Hopf derivative inequality is asserted at a
singular contact point.

\begin{remark}[Intrinsic formulation and pullback metric]
The state equation is intrinsically defined by the Riemannian metric
$g$. Indeed, the gradient is the metric dual of the differential, the
divergence is the trace of the covariant derivative, and $dV_g$ is the
Riemannian volume measure. Hence
\[
-\Delta_{g,p}u
=-\operatorname{div}_g\bigl(|\nabla_g u|_g^{p-2}\nabla_g u\bigr)=f
\]
is independent of any Euclidean background structure.

More importantly for the stability analysis, if
$\Phi:\Omega\to\Omega'$ is a diffeomorphism, the state equation on
$\Omega'$ can be pulled back to the fixed domain $\Omega$. The
resulting weak equation is expressed in terms of the pullback metric
$\Phi^*g$ and the transported source. Thus the stability argument is
naturally formulated on a fixed domain for a family of uniformly
equivalent Riemannian metrics.
\end{remark}

\subsection{The energy structure and the role of the source}

The functional
\[
\mathcal E_\Omega(v)
=
\frac1p\int_\Omega|\nabla_gv|_g^p\,dV_g
-\int_\Omega fv\,dV_g
\]
contains the complete variational structure of the state equation.
The first term is strictly convex on $W_0^{1,p}(\Omega)$, while the
source term is linear. Therefore the unique weak solution is not only
a critical point but the unique global minimizer of $\mathcal E_\Omega$.
This observation is useful in several places: it yields the energy
identity, gives uniform estimates, and provides a convenient route to
stability under perturbation of the metric and the domain.

The source term is assumed to be supported in the fixed compact set
$K$. Since every admissible domain contains $K$, the forcing does not
change when the boundary is deformed away from $K$. This is one of the
reasons for introducing a fixed core $C$ with $K\subset C$: the nonlinear
state problem can be compared on the common region $C$, while the
geometric part of the optimization is allowed to vary outside it.

If $f\geq0$, the weak maximum principle gives $u_\Omega\geq0$. When
$f\not\equiv0$ and $\Omega$ is connected, the strong maximum principle
implies
\[
u_\Omega>0\qquad\text{in }\Omega.
\]
Under the boundary regularity required below, Hopf's principle then
implies that at a regular boundary point
\[
\partial_{\nu_\Omega}u_\Omega<0.
\]
Thus the outward normal derivative has the sign compatible with the
geometric convention used in the free-boundary condition.

\subsection{Energy estimates and monotonicity with respect to domains}

\begin{lemma}[Uniform Poincar\'e inequality on $\mathcal A_C(M)$]
\label{lem:uniform-poincare}
There exists a constant $C_P>0$, depending only on $M$, $g$, $p$ and the
fixed set $C$, such that
\[
\|v\|_{L^p(\Omega)}\le C_P\|\nabla_g v\|_{L^p(\Omega)}
\qquad\text{for every }\Omega\in\mathcal A_C(M)
\text{ and every }v\in W_0^{1,p}(\Omega).
\]
\end{lemma}

\begin{proof}
Suppose, by contradiction, that no such uniform constant exists. Then there
are $\Omega_j\in\mathcal A_C(M)$ and $v_j\in W_0^{1,p}(\Omega_j)$ such that
\[
\|v_j\|_{L^p(\Omega_j)}=1,
\qquad
\|\nabla_g v_j\|_{L^p(\Omega_j)}\longrightarrow0.
\]
By the compactness theorem for the admissible class, after extraction we may
assume that $\Omega_j\to\Omega$ in the strong $RC$-$GNP$ topology, with
$\Omega\in\mathcal A_C(M)$, and there are $C^1$ diffeomorphisms
$\Phi_j:M\to M$ such that $\Phi_j(\Omega)=\Omega_j$ and
$\Phi_j\to\operatorname{Id}$ in $C^1(M)$. Set
\[
g_j=\Phi_j^*g,
\qquad
w_j=v_j\circ\Phi_j.
\]
Then $w_j\in W_0^{1,p}(\Omega)$ and, by the change-of-variables identity,
\[
\|w_j\|_{L^p(\Omega,g_j)}=1,
\qquad
\|\nabla_{g_j}w_j\|_{L^p(\Omega,g_j)}
=\|\nabla_gv_j\|_{L^p(\Omega_j,g)}\longrightarrow0.
\]
Since $g_j\to g$ uniformly and the metrics are uniformly equivalent, $(w_j)$
is bounded in $W^{1,p}(\Omega)$. After extraction, $w_j\rightharpoonup w$
weakly in $W^{1,p}(\Omega)$ and strongly in $L^p(\Omega)$. The uniform
convergence $g_j\to g$ and the vanishing gradient norms imply
$\nabla_gw=0$, so $w$ is constant on the connected domain $\Omega$.
Because $w_j\in W_0^{1,p}(\Omega)$, the weak limit also belongs to
$W_0^{1,p}(\Omega)$; hence this constant is zero. Thus
$w_j\to0$ strongly in $L^p(\Omega)$.
On the other hand, uniform convergence of the volume densities gives
\[
1=\|w_j\|_{L^p(\Omega,g_j)}^p
\longrightarrow \|w\|_{L^p(\Omega,g)}^p=0,
\]
a contradiction.\end{proof}

The state equation also yields estimates that are uniform over bounded
families of admissible domains. Indeed, testing with $u_\Omega$ gives
\[
\|\nabla_g u_\Omega\|_{L^p(\Omega)}^p
=\int_\Omega fu_\Omega\,dV_g.
\]
By Hölder's inequality and Lemma~\ref{lem:uniform-poincare},
\[
\|\nabla_g u_\Omega\|_{L^p(\Omega)}^p
\leq
\|f\|_{L^{p'}(\Omega)}
\|u_\Omega\|_{L^p(\Omega)}
\leq
C_P\|f\|_{L^{p'}(M)}
\|\nabla_g u_\Omega\|_{L^p(\Omega)}.
\]
Hence
\[
\|\nabla_g u_\Omega\|_{L^p(\Omega)}
\leq
C\|f\|_{L^{p'}(M)}^{1/(p-1)},
\]
where the constant is uniform on a uniformly controlled admissible
family. In particular,
\[
T_{p,g}(\Omega)
\leq
C\|f\|_{L^{p'}(M)}^{p'}.
\]

There is also a useful monotonicity property. If
$\Omega_1\subset\Omega_2$ and both domains contain $K$, comparison gives
\[
0\leq u_{\Omega_1}\leq u_{\Omega_2}
\qquad\text{in }\Omega_1.
\]
Since $f$ is supported in $K\subset\Omega_1$,
\[
T_{p,g}(\Omega_1)
=\int_Kfu_{\Omega_1}\,dV_g
\leq
\int_Kfu_{\Omega_2}\,dV_g
=T_{p,g}(\Omega_2).
\]
Thus the $p$-torsion is monotone increasing under inclusion. This
property explains the competition in $\mathcal J_{p,g}$: enlarging a
domain improves the torsional contribution but increases perimeter and
volume.

The $RC$-$GNP$ assumptions provide both geometric compactness of the domains and the ambient diffeomorphisms needed to identify the varying Sobolev spaces. The uniform defining-function bounds and tubular radius prevent boundary degeneration and allow the normal graph construction used below. After pullback, the metrics $g_j=\Phi_j^*g$ converge to $g$ uniformly, which is the geometric input required by the monotonicity argument.

\subsection{A useful pullback identity}

Let $\Phi:M\to M$ be a $C^1$ diffeomorphism and let $g_\Phi=\Phi^*g$.
For a function $v$ on a domain $\Omega$, the change of variables
$y=\Phi(x)$ gives
\[
\int_{\Phi(\Omega)}|\nabla_gv|_g^p\,dV_g
=
\int_\Omega|\nabla_{g_\Phi}(v\circ\Phi)|_{g_\Phi}^p\,dV_{g_\Phi}.
\]
Similarly,
\[
\int_{\Phi(\Omega)}fv\,dV_g
=
\int_\Omega(f\circ\Phi)(v\circ\Phi)\,dV_{g_\Phi}.
\]
Thus the weak state equation on the moving domain is exactly equivalent
to a weak equation on the fixed domain with a varying metric and a
transported source. This identity is the geometric core of the
stability theorem proved below.


\subsection{Regularity}

The weak theory is based on $f\in L^{p'}$. For the pointwise boundary
condition and shape-differentiation arguments we impose stronger
assumptions. In particular, $f\in L^{p'}(M)$ is sufficient for weak
existence and stability, whereas the boundary calculations below are
made only under additional regularity assumptions ensuring
$u_\Omega\in C^{1,\beta}(\overline{\Omega})$ for some $\beta\in(0,1)$.
We therefore treat $u_\Omega\in C^{1,\beta}(\overline{\Omega})$ as an
explicit hypothesis whenever a pointwise normal derivative or a classical
shape derivative is used; it is not silently deduced from the weak-level
assumptions.

Since $u_\Omega\in C^{1,\beta}(\overline{\Omega})$ and
$u_\Omega=0$ on $\partial\Omega$, the tangential derivatives of
$u_\Omega$ vanish on $\partial\Omega$. Hence, at every
$x\in\partial\Omega$,
\[
\nabla_g u_\Omega(x)
=
(\partial_{\nu_\Omega}u_\Omega)(x)\,\nu_\Omega(x),
\]
and therefore
\[
|\nabla_g u_\Omega(x)|_g
=
|\partial_{\nu_\Omega}u_\Omega(x)|.
\]


\subsection{Well-posedness on a fixed domain}

Before studying domain perturbations, it is useful to isolate the
well-posedness mechanism for the state equation. Let $\Omega\Subset M$ be
an admissible domain and define
\[
\mathcal E_\Omega(v)
=
\frac1p\int_\Omega |\nabla_gv|_g^p\,dV_g
-\int_\Omega fv\,dV_g,
\qquad v\in W_0^{1,p}(\Omega).
\]
The Sobolev space $W_0^{1,p}(\Omega)$ is reflexive for
$1<p<\infty$, and Poincar\'e's inequality gives
\[
\|v\|_{L^p(\Omega)}
\leq C_P\|\nabla_gv\|_{L^p(\Omega)}.
\]
Consequently, by H\"older's inequality,
\[
\left|\int_\Omega fv\,dV_g\right|
\leq
C_P\|f\|_{L^{p'}(\Omega)}
\|\nabla_gv\|_{L^p(\Omega)}.
\]
It follows that $\mathcal E_\Omega$ is coercive:
\[
\mathcal E_\Omega(v)
\geq
\frac1p\|\nabla_gv\|_{L^p(\Omega)}^p
-
C\|f\|_{L^{p'}(\Omega)}
\|\nabla_gv\|_{L^p(\Omega)}.
\]
Since the first term is strictly convex, the direct method gives a
unique minimizer $u_\Omega$. Its Euler equation is
\[
\int_\Omega
|\nabla_gu_\Omega|_g^{p-2}
g(\nabla_gu_\Omega,\nabla_g\varphi)\,dV_g
=
\int_\Omega f\varphi\,dV_g
\]
for every $\varphi\in W_0^{1,p}(\Omega)$.

The strict monotonicity used throughout the paper follows from the
pointwise convexity of $\xi\mapsto|\xi|^p/p$. More precisely,
\[
\bigl(|\xi|^{p-2}\xi-|\eta|^{p-2}\eta\bigr)\cdot(\xi-\eta)>0
\]
whenever $\xi\neq\eta$. Applied fibrewise with the metric $g$, this
identity proves uniqueness directly. It also explains why the argument
extends without change from Euclidean gradients to Riemannian gradients:
at each point of $M$, the tangent space is a finite-dimensional inner
product space and the nonlinear constitutive law is the same radial map.

There is an additional useful consequence. If $f_1,f_2\in L^{p'}(M)$ and
$u_1,u_2$ are the corresponding states on the same domain, then testing
the difference of the two equations by $u_1-u_2$ yields
\[
\int_\Omega
\bigl(A(\nabla_gu_1)-A(\nabla_gu_2)\bigr)
\cdot_g(\nabla_gu_1-\nabla_gu_2)\,dV_g
=
\int_\Omega(f_1-f_2)(u_1-u_2)\,dV_g.
\]
For $p\ge2$ this gives the standard estimate
\[
\|\nabla_g(u_1-u_2)\|_{L^p(\Omega)}
\leq
C\|f_1-f_2\|_{L^{p'}(\Omega)}^{1/(p-1)},
\]
while for $1<p<2$ the corresponding weighted monotonicity estimate is
used. Thus the fixed-domain problem has exactly the stability structure
needed after a geometric pullback.

\subsection{Uniform equivalence after pullback}

The pullback construction in the stability theorem replaces a moving
domain by a fixed domain endowed with the metric $g_j=\Phi_j^*g$.
Because $\Phi_j\to\operatorname{Id}$ in $C^1(M)$, the metric coefficients
and their inverses converge uniformly. In particular, there exist
constants $0<c\le C<\infty$ independent of $j$ such that
\[
c\,g(\xi,\xi)
\leq g_j(\xi,\xi)
\leq C\,g(\xi,\xi)
\]
for every $\xi\in TM$. Hence
\[
c_1|\xi|_g\leq|\xi|_{g_j}\leq C_1|\xi|_g
\]
and
\[
c_2\,dV_g\leq dV_{g_j}\leq C_2\,dV_g.
\]
The constants depend only on a uniform $C^1$ bound for the
diffeomorphisms and their inverses.

As a consequence, all Sobolev norms induced by $g$ and $g_j$ on the fixed
domain $\Omega$ are uniformly equivalent:
\[
C^{-1}\|v\|_{W^{1,p}(\Omega,g)}
\leq
\|v\|_{W^{1,p}(\Omega,g_j)}
\leq
C\|v\|_{W^{1,p}(\Omega,g)}.
\]
The same observation applies to the duality pairing with
$L^{p'}(\Omega)$. This uniform equivalence is the precise reason why the
coercivity and Poincar\'e constants in the pulled-back problems can be
chosen independently of $j$.

The convergence $g_j\to g$ is stronger than mere equivalence. For fixed
$w,\varphi\in W_0^{1,p}(\Omega)$,
\[
B_j(w,\varphi)\longrightarrow B(w,\varphi),
\]
where
\[
B(w,\varphi)
=
\int_\Omega
|\nabla_gw|_g^{p-2}
g(\nabla_gw,\nabla_g\varphi)\,dV_g.
\]
Indeed, the integrands converge pointwise after approximation by smooth
functions, while the uniform metric bounds provide an integrable
majorant. This observation is the coefficient-convergence ingredient in
the Minty argument.

\section{Stability theorem}

\begin{lemma}[Continuity of composition under $C^1$ convergence]
\label{lem:composition-C1-W1p}
Let $(M,g)$ be a compact smooth Riemannian manifold and let $1<p<\infty$. Suppose that $\Phi_j:M\to M$ are $C^1$ diffeomorphisms such that
\[
\Phi_j\to\operatorname{Id}\quad\text{in }C^1(M).
\]
Then, for every $w\in W^{1,p}(M)$,
\[
w\circ\Phi_j\to w\quad\text{strongly in }W^{1,p}(M),
\]
and the same conclusion holds for $\Phi_j^{-1}$. Moreover, the composition
operators induced by $\Phi_j$ and $\Phi_j^{-1}$ are uniformly bounded on
$W^{1,p}(M)$. More generally, for every $1\le q<\infty$ and every
$h\in L^q(M)$,
\[
h\circ\Phi_j\to h\quad	ext{strongly in }L^q(M),
\]
and the same holds for $\Phi_j^{-1}$.
\end{lemma}

\begin{proof}
The convergence $\Phi_j\to\operatorname{Id}$ in $C^1(M)$ implies, for all
sufficiently large $j$, uniform bounds on $D\Phi_j$ and $D\Phi_j^{-1}$.
The change-of-variables formula and the chain rule therefore give
\[
\|w\circ\Phi_j\|_{W^{1,p}(M)}\le C\|w\|_{W^{1,p}(M)},
\]
with $C$ independent of $j$, and similarly for $\Phi_j^{-1}$.

If $w\in C^\infty(M)$, then $w\circ\Phi_j\to w$ in $C^1(M)$, hence
in $W^{1,p}(M)$. For general $w\in W^{1,p}(M)$, choose
$w_\varepsilon\in C^\infty(M)$ with
$\|w-w_\varepsilon\|_{W^{1,p}(M)}<\varepsilon$. Then
\begin{align*}
\|w\circ\Phi_j-w\|_{W^{1,p}(M)}
&\le (C+1)\varepsilon
+\|w_\varepsilon\circ\Phi_j-w_\varepsilon\|_{W^{1,p}(M)}.
\end{align*}
Taking first $j\to\infty$ and then $\varepsilon\to0$ proves the claim.
Since $\Phi_j^{-1}\to\operatorname{Id}$ in $C^1(M)$ as well, the same argument
applies to the inverses. The $L^q$ statement follows by the same density
argument, first for continuous functions and then for arbitrary
$L^q$ functions.
\end{proof}

\begin{theorem}[Stability of the Riemannian $p$-Laplace problem]
\label{thm:stability-p-laplace}
Let $\Omega_j,\Omega\in\mathcal A_C(M)$ and assume that
\[
\Omega_j\longrightarrow\Omega
\]
in the strong $RC$-$GNP$ topology. Let
$f\in L^{p'}(M)$, where $p'=p/(p-1)$, and let
$u_j=u_{\Omega_j}$ and $u=u_\Omega$ be the corresponding weak
solutions of the Riemannian $p$-Laplace Dirichlet problem, extended by
zero to $M$. Then
\[
u_j\longrightarrow u
\qquad\text{strongly in }W^{1,p}(M).
\]
In particular,
\[
\int_M|\nabla_g u_j|_g^p\,dV_g
\longrightarrow
\int_M|\nabla_g u|_g^p\,dV_g.
\]
\end{theorem}

\begin{proof}
Let $\Phi_j:M\to M$ be the diffeomorphisms associated with the strong
$RC$-$GNP$ convergence, so that
\[
\Phi_j(\Omega)=\Omega_j,
\qquad
\Phi_j\to\operatorname{Id}
\quad\text{in }C^1(M).
\]
Set
\[
g_j:=\Phi_j^*g.
\]
Then
\[
g_j\to g\qquad\text{uniformly on }M,
\]
and, since $\Phi_j\to\operatorname{Id}$ in $C^1$ and the $\Phi_j$ are
diffeomorphisms, the metrics $g_j$ are uniformly equivalent to $g$.
In particular, there exists $c\geq1$, independent of $j$, such that
\[
c^{-1}|\xi|_g\leq |\xi|_{g_j}\leq c|\xi|_g
\]
for every $\xi\in TM$.

Define
\[
v_j=u_j\circ\Phi_j\qquad\text{in }\Omega.
\]
Then
\[
v_j\in W_0^{1,p}(\Omega).
\]
After the change of variables $y=\Phi_j(x)$, the weak formulation on
$\Omega_j$ becomes
\begin{equation}
\int_\Omega
|\nabla_{g_j}v_j|_{g_j}^{p-2}
g_j(\nabla_{g_j}v_j,\nabla_{g_j}\varphi)
\,dV_{g_j}
=
\int_\Omega (f\circ\Phi_j)\varphi\,dV_{g_j}
\label{eq:transported-p-laplace}
\end{equation}
for every $\varphi\in W_0^{1,p}(\Omega)$.

For later use, define the nonlinear form
\[
B_j(w,\varphi)
=
\int_\Omega
|\nabla_{g_j}w|_{g_j}^{p-2}
g_j(\nabla_{g_j}w,\nabla_{g_j}\varphi)
\,dV_{g_j}
\]
and
\[
F_j(\varphi)
=
\int_\Omega(f\circ\Phi_j)\varphi\,dV_{g_j}.
\]
The uniform equivalence of $g_j$ and $g$, together with the uniform
Poincar\'e inequality on the fixed domain $\Omega$, gives
\[
\|v_j\|_{W^{1,p}(\Omega)}\leq C.
\]
Indeed, taking $\varphi=v_j$ in \eqref{eq:transported-p-laplace},
using H\"older's inequality and Poincar\'e's inequality, yields the
bound with $C$ independent of $j$.

Since $W_0^{1,p}(\Omega)$ is reflexive, after extraction of a
subsequence there exists $v\in W_0^{1,p}(\Omega)$ such that
\[
v_j\rightharpoonup v
\qquad\text{weakly in }W_0^{1,p}(\Omega).
\]
Moreover, by the Rellich--Kondrachov theorem,
\[
v_j\longrightarrow v
\qquad\text{strongly in }L^p(\Omega).
\]

We next identify the weak limit. The passage to the limit in the nonlinear term cannot be justified by weak convergence alone. We therefore use the monotonicity of the transported operators.

For an arbitrary $w{ \in }W_0^{1,p}(\Omega)$, monotonicity gives
\[
0\leq B_j(v_j,v_j-w)-B_j(w,v_j-w).
\]
The first term equals $F_j(v_j-w)$ by the transported weak formulation.
By the $L^{p'}$ version of Lemma~\ref{lem:composition-C1-W1p},
\[
f\circ\Phi_j\longrightarrow f
\qquad\text{strongly in }L^{p'}(M).
\]
Together with $v_j\to v$ strongly in $L^p(\Omega)$ and the uniform
convergence of the volume densities $dV_{g_j}$ to $dV_g$, this gives
\[
F_j(v_j-w)\longrightarrow \int_\Omega f(v-w)\,dV_g.
\]
For the second term, $w$ is fixed and the coefficients of the transported
operator converge uniformly. Hence the weak convergence of $v_j$ gives
\[
B_j(w,v_j-w)\longrightarrow
\int_\Omega
|\nabla_g w|_g^{p-2}
 g(\nabla_g w,\nabla_g(v-w))\,dV_g.
\]
It follows that, for every $w\in W_0^{1,p}(\Omega)$,
\[
\int_\Omega
|\nabla_g w|_g^{p-2}
 g(\nabla_g w,\nabla_g(w-v))\,dV_g
\geq
\int_\Omega f(w-v)\,dV_g.
\]
This is the Minty variational inequality for the limiting operator. Taking
$w=v-t\varphi$, with $t>0$ and arbitrary $\varphi\in W_0^{1,p}(\Omega)$,
and then dividing by $t$, we obtain, as $t\downarrow0$,
\[
-\int_\Omega
|\nabla_g(v-t\varphi)|_g^{p-2}
 g(\nabla_g(v-t\varphi),\nabla_g\varphi)\,dV_g
\geq
-\int_\Omega f\varphi\,dV_g.
\]
Using the continuity of the map
$\xi\mapsto|\xi|_g^{p-2}\xi$ from $L^p$ into $L^{p'}$, we obtain
\[
\int_\Omega
|\nabla_g v|_g^{p-2}
 g(\nabla_g v,\nabla_g\varphi)\,dV_g
\leq
\int_\Omega f\varphi\,dV_g.
\]
Repeating the argument with $w=v+t\varphi$ yields the reverse inequality.
Therefore
\[
\int_\Omega
|\nabla_g v|_g^{p-2}
 g(\nabla_g v,\nabla_g\varphi)\,dV_g
=
\int_\Omega f\varphi\,dV_g
\]
for every $\varphi\in W_0^{1,p}(\Omega)$. Thus $v$ is a weak solution of
the limiting Dirichlet problem. By uniqueness,
\[
v=u.
\]
Consequently,
\[
v_j\rightharpoonup u
\qquad\text{weakly in }W_0^{1,p}(\Omega).
\]

It remains to prove strong convergence. Taking $\varphi=v_j-u$ in
\eqref{eq:transported-p-laplace} gives
\[
B_j(v_j,v_j-u)=F_j(v_j-u)\longrightarrow0.
\]
Moreover, since $u$ is fixed and $g_j\to g$ uniformly,
\[
B_j(u,v_j-u)\longrightarrow0
\]
by the weak convergence of $v_j$ to $u$. Hence
\begin{equation}
B_j(v_j,v_j-u)-B_j(u,v_j-u)\longrightarrow0.
\label{eq:monotonicity-limit}
\end{equation}

For $p\ge2$, the uniform monotonicity inequality gives, for some
$c_p>0$ independent of $j$,
\[
\left(
|\xi|_{g_j}^{p-2}\xi-|\eta|_{g_j}^{p-2}\eta
\right)\cdot_{g_j}(\xi-\eta)
\ge c_p|\xi-\eta|_{g_j}^p.
\]
Using \eqref{eq:monotonicity-limit}, we obtain
\[
\|\nabla_{g_j}v_j-\nabla_{g_j}u\|_{L^p(\Omega,g_j)}\longrightarrow0.
\]

For $1<p<2$, we use the standard weighted monotonicity estimate.
The quotient below is understood to be $0$ when $\xi=\eta=0$.
For $1<p<2$, we use the uniform estimate
\[
\left(
|\xi|_{g_j}^{p-2}\xi-|\eta|_{g_j}^{p-2}\eta
\right)\cdot_{g_j}(\xi-\eta)
\ge c_p
\frac{|\xi-\eta|_{g_j}^2}
{(|\xi|_{g_j}+|\eta|_{g_j})^{2-p}}.
\]
It follows from \eqref{eq:monotonicity-limit} that
\[
\int_\Omega
\frac{|\nabla_{g_j}v_j-\nabla_{g_j}u|_{g_j}^2}
{(|\nabla_{g_j}v_j|_{g_j}+|\nabla_{g_j}u|_{g_j})^{2-p}}
\,dV_{g_j}\longrightarrow0.
\]
By H\"older's inequality,
\begin{align*}
\int_\Omega
|\nabla_{g_j}v_j-\nabla_{g_j}u|_{g_j}^p\,dV_{g_j}
&\leq
\left(
\int_\Omega
\frac{|\nabla_{g_j}v_j-\nabla_{g_j}u|_{g_j}^2}
{(|\nabla_{g_j}v_j|_{g_j}+|\nabla_{g_j}u|_{g_j})^{2-p}}
\,dV_{g_j}
\right)^{p/2}
\\
&\quad\times
\left(
\int_\Omega
(|\nabla_{g_j}v_j|_{g_j}+|\nabla_{g_j}u|_{g_j})^p
\,dV_{g_j}
\right)^{(2-p)/2}.
\end{align*}
The second factor is uniformly bounded, so the first one implies
\[
\|\nabla_{g_j}v_j-\nabla_{g_j}u\|_{L^p(\Omega,g_j)}\longrightarrow0.
\]
Thus, for every $1<p<\infty$,
\[
v_j\longrightarrow u
\qquad\text{strongly in }W_0^{1,p}(\Omega).
\]

Let $\widetilde v_j$ and $\widetilde u$ denote the zero extensions of
$v_j$ and $u$ from $\Omega$ to $M$. Since
$v_j\to u$ strongly in $W_0^{1,p}(\Omega)$, we have
\[
\widetilde v_j\to\widetilde u\quad\text{strongly in }W^{1,p}(M).
\]
By construction, $\widetilde v_j=\widetilde u_j\circ\Phi_j$ on $M$,
where $\widetilde u_j$ is the zero extension of $u_j$. Hence
$\widetilde u_j=\widetilde v_j\circ\Phi_j^{-1}$. Therefore, by
Lemma~\ref{lem:composition-C1-W1p},
\begin{align*}
\|\widetilde u_j-\widetilde u\|_{W^{1,p}(M)}
&\le C\|\widetilde v_j-\widetilde u\|_{W^{1,p}(M)}
+\|\widetilde u\circ\Phi_j^{-1}-\widetilde u\|_{W^{1,p}(M)}
\longrightarrow0.
\end{align*}
Thus $u_j\to u$ strongly in $W^{1,p}(M)$.
In particular,
\[
\|\nabla_g u_j\|_{L^p(M)}
\longrightarrow
\|\nabla_g u\|_{L^p(M)},
\]
and hence
\[
\int_M|\nabla_g u_j|_g^p\,dV_g
\longrightarrow
\int_M|\nabla_g u|_g^p\,dV_g.
\]

\end{proof}


\subsection{Interpretation of the stability mechanism}

The stability theorem can be separated into three logically distinct
steps. The first is geometric: strong $RC$-$GNP$ convergence supplies
diffeomorphisms $\Phi_j$ that identify the varying domains with one fixed
domain. The second is functional analytic: the pulled-back states are
uniformly bounded in the reflexive space $W_0^{1,p}(\Omega)$ and therefore
admit weakly convergent subsequences. The third is nonlinear: strict
monotonicity and the Minty argument identify the weak limit with the
unique solution of the limiting equation, after which quantitative
monotonicity yields strong convergence.

This separation is useful because the first two steps do not depend on
the distinction between the cases $p<2$ and $p\ge2$. The distinction
appears only in the final upgrade from weak to strong convergence. For
$p\ge2$, the map $\xi\mapsto|\xi|^{p-2}\xi$ is uniformly monotone of
power type $p$. For $1<p<2$, the map is singular at $\xi=0$, and the
appropriate estimate contains the weight
\[
\bigl(|\xi|+|\eta|\bigr)^{2-p}.
\]
The uniform energy bound controls the weighted factor and allows
H\"older's inequality to recover convergence in $L^p$.

The strong convergence has two immediate consequences that are important
for optimization. First,
\[
\int_\Omega|\nabla_{g_j}v_j|_{g_j}^p\,dV_{g_j}
\longrightarrow
\int_\Omega|\nabla_gu|_g^p\,dV_g.
\]
Second, because the zero extensions are related by the ambient
diffeomorphisms, the corresponding states converge strongly on the
original manifold. Thus the analytic state functional does not lose
mass in the limit even though the domains themselves vary.

\subsection{Energy identity and dual characterization of the torsion}

Testing the state equation by $u_\Omega$ gives
\[
T_{p,g}(\Omega)
=
\int_\Omega|\nabla_gu_\Omega|_g^p\,dV_g
=
\int_\Omega fu_\Omega\,dV_g.
\]
Consequently,
\[
\mathcal E_\Omega(u_\Omega)
=
-\frac1{p'}T_{p,g}(\Omega).
\]
Equivalently, the torsion admits the variational characterization
\[
T_{p,g}(\Omega)
=
p'\sup_{v\in W_0^{1,p}(\Omega)}
\left\{
\int_\Omega fv\,dV_g
-\frac1p\int_\Omega|\nabla_gv|_g^p\,dV_g
\right\}.
\]
Indeed, the expression in braces is maximized precisely at
$u_\Omega$. This characterization is useful because it makes the
competition between the nonlinear state energy and the geometric
penalties transparent.

The scaling of the nonlinear energy also explains the occurrence of the
conjugate exponent $p'$. If the source is multiplied by a positive
constant $a$, then the state is multiplied by $a^{1/(p-1)}$:
\[
u_\Omega[af]
=
a^{1/(p-1)}u_\Omega[f].
\]
Hence
\[
T_{p,g}(\Omega;af)
=
a^{p'}T_{p,g}(\Omega;f).
\]
This scaling is consistent with the exponent appearing in the explicit
radial formulas on the sphere and provides a useful check on all
dimension-dependent calculations.

\section{Consequence for the $p$-torsion}

\begin{definition}[$p$-torsion]
\label{def:p-torsion}
Let $\Omega\in\mathcal A_C(M)$, and let $u_\Omega\in W_0^{1,p}(\Omega)$
be the unique weak solution of
\[
-\Delta_{g,p}u_\Omega=f\quad\text{in }\Omega,
\qquad
u_\Omega=0\quad\text{on }\partial\Omega.
\]
The \emph{$p$-torsion} of $\Omega$ associated with the fixed source $f$ is
\[
\boxed{
T_{p,g}(\Omega):=\int_\Omega|\nabla_g u_\Omega|_g^p\,dV_g.
}
\]
Testing the weak equation with $u_\Omega$ gives the equivalent identity
\[
T_{p,g}(\Omega)=\int_\Omega f u_\Omega\,dV_g.
\]
Hence $T_{p,g}(\Omega)$ is a well-defined finite domain functional under
the standing assumption $f\in L^{p'}(M)$.
\end{definition}

\begin{corollary}[Continuity of the $p$-torsion]
\label{cor:continuity-p-torsion}
Let $\Omega_j,\Omega\in\mathcal A_C(M)$ and assume that
$\Omega_j\to\Omega$ in the strong $RC$-$GNP$ topology. Let $u_j$ and
$u$ denote the corresponding weak solutions. Then
\[
T_{p,g}(\Omega_j)\longrightarrow T_{p,g}(\Omega).
\]
\end{corollary}

\begin{proof}
By Theorem~\ref{thm:stability-p-laplace}, the zero extensions of the
states satisfy
\[
u_j\longrightarrow u\qquad\text{strongly in }W^{1,p}(M).
\]
In particular,
\[
\int_M|\nabla_g u_j|_g^p\,dV_g
\longrightarrow
\int_M|\nabla_g u|_g^p\,dV_g.
\]
Because the zero extensions vanish outside their respective domains,
\[
\int_M|\nabla_g u_j|_g^p\,dV_g=T_{p,g}(\Omega_j),
\qquad
\int_M|\nabla_g u|_g^p\,dV_g=T_{p,g}(\Omega).
\]
Therefore $T_{p,g}(\Omega_j)\to T_{p,g}(\Omega)$.
\end{proof}

We consider
\[
\boxed{
\mathcal J_{p,g}(\Omega)
=
-T_{p,g}(\Omega)
+\sigma P_g(\Omega)
+k^pV_g(\Omega).
}
\]

\subsection{Continuity of the geometric quantities}

We record two elementary consequences of strong $RC$-$GNP$ convergence.
If $\Omega_j\to\Omega$ and the associated diffeomorphisms satisfy
$\Phi_j\to\operatorname{Id}$ in $C^1$, then the volume densities of
$g_j=\Phi_j^*g$ converge uniformly to those of $g$. Consequently,
\[
V_g(\Omega_j)
=V_{g_j}(\Omega)
\longrightarrow V_g(\Omega).
\]
The same conclusion follows directly from the $L^1$ convergence of
characteristic functions established in the compactness theorem.

For the boundary geometry, the $C^2$ convergence of the defining
functions implies convergence of the tangent spaces, unit normals and
second fundamental forms on corresponding boundary points. In
particular, for the normal graph representation
\[
\partial\Omega_j
=
\{\exp_x(\eta_j(x)\nu_\Omega(x)):x\in\partial\Omega\},
\]
we have $\eta_j\to0$ in $C^2(\partial\Omega)$. We make the resulting
perimeter convergence explicit. Define
\[
\Psi_j:\partial\Omega\longrightarrow\partial\Omega_j,
\qquad
\Psi_j(x)=\exp_x\bigl(\eta_j(x)\nu_\Omega(x)\bigr).
\]
Since $\eta_j\to0$ in $C^2(\partial\Omega)$ and the ambient metric is
smooth, the pullback of the induced metric satisfies
\[
\Psi_j^*\bigl(g|_{\partial\Omega_j}\bigr)
\longrightarrow g|_{\partial\Omega}
\qquad\text{uniformly on }\partial\Omega.
\]
Equivalently, the corresponding tangential Jacobians satisfy
\[
J_{\partial\Omega}\Psi_j\longrightarrow1
\qquad\text{uniformly on }\partial\Omega.
\]
Therefore the change-of-variables formula on the hypersurfaces gives
\[
P_g(\Omega_j)
=
\int_{\partial\Omega_j}1\,dS_g
=
\int_{\partial\Omega}J_{\partial\Omega}\Psi_j\,dS_g
\longrightarrow
\int_{\partial\Omega}1\,dS_g
=P_g(\Omega).
\]
Thus, under strong $RC$-$GNP$ convergence,
\[
P_g(\Omega_j)\longrightarrow P_g(\Omega).
\]
The same normal-graph representation also yields uniform control of the
mean curvatures, which is the geometric input used in the local contact
analysis.

\subsection{Why the compactness and stability hypotheses are matched}

The choice of the $RC$-$GNP$ topology is not merely technical. The
existence proof requires a convergent subsequence of admissible domains,
whereas the nonlinear state equation requires a common functional space
on which weak convergence can be tested. If one only knew that
$\chi_{\Omega_j}\to\chi_\Omega$ in $L^1$, the functions
$u_{\Omega_j}$ would still live on different spaces
$W_0^{1,p}(\Omega_j)$. The maps $\Phi_j$ solve this identification
problem.

After pullback, all states belong to $W_0^{1,p}(\Omega)$ and solve
problems associated with metrics $g_j$ converging uniformly to $g$.
The weak compactness of $W_0^{1,p}(\Omega)$ then produces a candidate
limit, and strict monotonicity identifies it uniquely. Finally,
quantitative monotonicity upgrades weak convergence to strong
convergence. 

\subsection{The case $p=2$ and the nonlinear extension}

When $p=2$, the operator reduces to the Laplace--Beltrami operator,
\[
\Delta_{g,2}u=\Delta_gu,
\]
and the torsion becomes the classical Riemannian torsional functional.
The state equation is then linear and the monotonicity argument reduces
to the usual coercivity of the Dirichlet form. The present paper retains
this classical structure but replaces linear elliptic theory by the
strictly convex $p$-energy for arbitrary $1<p<\infty$.

The nonlinear extension is not purely formal. In particular, the
boundary integrand in the Hadamard formula is
$|\nabla_g u|_g^p$, the natural stress tensor is nonlinear, and the
strong stability argument for $1<p<2$ requires a weighted monotonicity
estimate rather than the elementary Hilbert-space identity available
when $p=2$. These are the points where the $p$-Laplacian genuinely
changes the analysis.
All nonlinear formulas below reduce to the corresponding
Laplace--Beltrami formulas when $p=2$; this is an essential consistency
check on the exponents and signs.


\section{Main results}

The main results are formulated for the admissible class $\mathcal A_C(M)$ above.

\subsection{Existence of an optimal domain}

\begin{theorem}[Existence of a minimizer]
\label{thm:existence-minimizer}
Assume the hypotheses defining $\mathcal A_C(M)$, and let
$1<p<\infty$ and $f\in L^{p'}(M)$, $f\geq0$. Then the functional
$\mathcal J_{p,g}$ admits a minimizer
\[
\Omega^\ast\in \mathcal A_C(M):
\qquad
\mathcal J_{p,g}(\Omega^\ast)
=
\inf_{\Omega\in \mathcal A_C(M)}\mathcal J_{p,g}(\Omega).
\]
\end{theorem}

\begin{proof}
Since the energy estimate gives a uniform upper bound for $T_{p,g}(\Omega)$ on $\mathcal A_C(M)$, while $\sigma>0$ and $k>0$, we have
\[
\mathcal J_{p,g}(\Omega)
\geq
-C\|f\|_{L^{p'}(M)}^{p'}
\qquad\text{for all }\Omega\in\mathcal A_C(M).
\]
Hence the infimum is finite and there exists a minimizing sequence $(\Omega_j)\subset\mathcal A_C(M)$.

By Theorem~\ref{thm:compactness-RC-GNP}, after extraction we have $\Omega_j\to\Omega^*$ in the strong $RC$-$GNP$ topology for some $\Omega^*\in\mathcal A_C(M)$. The stability theorem gives $T_{p,g}(\Omega_j)\to T_{p,g}(\Omega^*)$. The $C^2$ convergence of the boundaries and the uniform tubular representation imply $P_g(\Omega_j)\to P_g(\Omega^*)$, while $L^1$ convergence of the characteristic functions gives $V_g(\Omega_j)\to V_g(\Omega^*)$. Hence $\mathcal J_{p,g}(\Omega_j)\to\mathcal J_{p,g}(\Omega^*)$, and $\Omega^*$ attains the infimum.\end{proof}

\subsection{Comparison with the reference domain}

\begin{proposition}[Comparison principle]
\label{prop:comparison}
Let $\Omega\in \mathcal A_C(M)$ and let $u_C$ and $u_\Omega$ solve the
$p$-Laplace Dirichlet problems on $C$ and $\Omega$, respectively.
Then
\[
0\leq u_C\leq u_\Omega
\qquad\text{in }C.
\]
\end{proposition}

\begin{proof}
Nonnegativity follows from the weak maximum principle since $f\ge0$ and both solutions have zero boundary data on their respective domains. To compare the two solutions, set $w=(u_C-u_\Omega)^+$. Since $u_C=0$ on $\partial C$ and $u_\Omega\ge0$ in $C$, the function $w$ belongs to $W_0^{1,p}(C)$. Its zero extension to $\Omega$ belongs to $W_0^{1,p}(\Omega)$, so it is admissible as a test function in the equation for $u_\Omega$. Using $w$ as a test function in the two weak equations gives
\[
\int_C\bigl(A(\nabla_g u_C)-A(\nabla_g u_\Omega)\bigr)\cdot_g\nabla_g w\,dV_g=0,
\qquad A(\xi)=|\xi|_g^{p-2}\xi.
\]
On $\{u_C>u_\Omega\}$ one has $\nabla_gw=\nabla_g u_C-\nabla_g u_\Omega$, so strict monotonicity of $A$ implies $\nabla_gw=0$ almost everywhere. Since $w\in W_0^{1,p}(C)$, $w=0$. Thus $u_C\le u_\Omega$ in $C$.\end{proof}

\begin{proposition}[Monotonicity of the $p$-torsion under inclusion]
\label{prop:torsion-monotonicity}
Let $\Omega_1,\Omega_2\in\mathcal A_C(M)$ satisfy
$\Omega_1\subset\Omega_2$, and let $u_1,u_2$ be the corresponding
weak solutions. Then
\[
0\le u_1\le u_2\qquad\text{in }\Omega_1,
\]
and consequently
\[
T_{p,g}(\Omega_1)\le T_{p,g}(\Omega_2).
\]
No strict inequality is asserted without additional hypotheses.
\end{proposition}

\begin{proof}
The comparison argument used in Proposition~\ref{prop:comparison} applies
with $C$ replaced by $\Omega_1$ and $\Omega$ replaced by $\Omega_2$; hence
$u_1\le u_2$ in $\Omega_1$. By the energy identity and $f\ge0$,
\[
T_{p,g}(\Omega_1)
=\int_{\Omega_1}fu_1\,dV_g
\le\int_{\Omega_1}fu_2\,dV_g
\le\int_{\Omega_2}fu_2\,dV_g
=T_{p,g}(\Omega_2).
\]
The last inequality uses $f\ge0$ and $u_2\ge0$.\end{proof}

\subsection{Optimality system}

For a minimizer $\Omega^\ast$, define
\[
\Gamma_0=\partial\Omega^\ast\cap\partial C
\]
and
\[
\Gamma=\partial\Omega^\ast\setminus\Gamma_0.
\]
At regular points of the free boundary, arbitrary small normal
variations are admissible.

\begin{theorem}[Free-boundary condition]
\label{thm:free-boundary-condition}
Assume the boundary and source regularity required for the classical
shape derivative, so that the Hadamard formula applies. Then at every
regular point of $\Gamma$,
\[
\boxed{
|\nabla_g u_{\Omega^\ast}|_g^p
=
\sigma H_{\partial\Omega^\ast}+k^p.
}
\]
\end{theorem}

\begin{proof}
Let $x_0\in\Gamma$ be a regular free-boundary point. Since $x_0$ is
separated from the fixed inclusion constraint $C$, choose a neighborhood
$U\subset M\setminus C$ of $x_0$. For every $X\in C_c^2(U;TM)$, the flows
of $X$ and $-X$ preserve admissibility for sufficiently small times.
Minimality of $\Omega^*$ therefore gives
\[
d\mathcal J_{p,g}(\Omega^*)[X]\ge0,
\qquad
d\mathcal J_{p,g}(\Omega^*)[-X]\ge0,
\]
and hence $d\mathcal J_{p,g}(\Omega^*)[X]=0$. By the Hadamard formula,
\[
0=\int_{\Gamma\cap U}
\left(\sigma H_{\partial\Omega^*}+k^p-
|\nabla_g u_{\Omega^*}|_g^p\right)g(X,\nu)\,dS_g.
\]
At a regular boundary point the normal trace $g(X,\nu)$ can be prescribed
arbitrarily on a smaller neighborhood of $x_0$. The fundamental lemma of
calculus of variations then implies that the integrand vanishes at $x_0$,
which gives the stated equality.
\end{proof}

\subsection{Curvature comparison at contact}

\begin{lemma}[Riemannian curvature comparison]
\label{lem:curvature-comparison}
Let $x_0\in\partial\Omega^*\cap\partial C$ be a regular tangential
contact point. Then
\[
\nu_{\Omega^*}(x_0)=\nu_C(x_0)
\]
and
\[
\mathrm{II}_{\partial\Omega^*}(x_0)
\leq
\mathrm{II}_{\partial C}(x_0)
\]
as quadratic forms on the common tangent space. Consequently,
\[
H_{\partial\Omega^*}(x_0)
\leq
H_{\partial C}(x_0).
\]
\end{lemma}

\begin{proof}
Choose geodesic normal coordinates centered at $x_0$ so that the common
tangent space is $\{x_n=0\}$ and the common outward normal is
$\partial_{x_n}$ at $x_0$. Since the two boundaries are tangent and
$C\subset\Omega^*$, after shrinking the coordinate neighborhood if
necessary, both boundaries can be written as graphs
\[
x_n=h_C(x')\qquad\text{and}\qquad x_n=h_*(x'),
\]
where $x'=(x_1,\ldots,x_{n-1})$, the interiors of both domains lie below
their respective graphs, and
\[
h_C(x')\le h_*(x')
\]
near $x_0$. At the contact point we have
\[
h_C(x_0')=h_*(x_0'),
\qquad
Dh_C(x_0')=Dh_*(x_0')=0,
\]
so $h_*-h_C$ has a local minimum at $x_0'$. Consequently,
\[
D^2h_*(x_0')-D^2h_C(x_0')\ge0
\]
as quadratic forms on the common tangent space.
It remains to keep track of the sign relating the graph Hessian to the
second fundamental form. With the convention
\[
\mathrm{II}(X,Y)=g(\nabla_X\nu,Y),
\]
and with the outward normal pointing in the $+x_n$ direction, the
Christoffel symbols vanish at $x_0$ in geodesic normal coordinates and
$Dh_C(x_0')=Dh_*(x_0')=0$. Hence, for tangent vectors
$X,Y\in T_{x_0}\partial C=T_{x_0}\partial\Omega^*$,
\[
\mathrm{II}_{\partial C}(X,Y)=-D^2h_C(x_0')[X,Y],
\qquad
\mathrm{II}_{\partial\Omega^*}(X,Y)=-D^2h_*(x_0')[X,Y].
\]
Therefore the Hessian inequality above gives
\[
\mathrm{II}_{\partial\Omega^*}(x_0)
\le
\mathrm{II}_{\partial C}(x_0)
\]
as quadratic forms. Taking the trace with respect to the common induced
metric yields
\[
H_{\partial\Omega^*}(x_0)
\le
H_{\partial C}(x_0).
\]
\end{proof}

\subsection{A sufficient condition for a free-boundary solution}

\begin{theorem}[Sufficient condition]
\label{thm:sufficient-condition}
Assume that $f\in C^{1,\alpha}(M)$ for some $\alpha\in(0,1)$, with
$f\geq0$ and $f\not\equiv0$. Let $u_C$ be the solution in $C$.
Assume in addition that, for some $\beta\in(0,1)$,
$u_C\in C^{1,\beta}(\overline C)$ and
$u_{\Omega^*}\in C^{1,\beta}(\overline{\Omega^*})$.
Assume that every point of $\partial C\cap\partial\Omega^*$, if nonempty,
is a regular tangential contact point and satisfies the interior sphere
condition required by the Hopf boundary point lemma.
Assume moreover that, at every such contact point, the strong
comparison principle applies to the ordered pair
$(u_{\Omega^*},u_C)$ and that the corresponding Hopf boundary point
principle yields the strict normal derivative inequality.
Finally, assume
\[
\boxed{
|\nabla_g u_C|_g^p
>
\sigma H_{\partial C}+k^p
\qquad\text{on }\partial C.
}
\]
Then $\Gamma_0=\partial\Omega^*\cap\partial C$ is empty. In particular,
$\partial\Omega^*$ is entirely free and the free-boundary Euler--Lagrange
condition holds on all of $\partial\Omega^*$:
\[
\boxed{
|\nabla_g u_{\Omega^*}|_g^p
=
\sigma H_{\partial\Omega^*}+k^p
\qquad\text{on }\partial\Omega^*.
}
\]
Thus $\Omega^*$ together with $u_{\Omega^*}$ solves the free-boundary problem
stated in the introduction.
\end{theorem}

\begin{proof}
Suppose, for contradiction, that $x_0\in\Gamma_0$.

First consider the case $\Omega^*=C$. Then
$u_{\Omega^*}=u_C$ and the contact optimality inequality at $x_0$
gives
\[
|\nabla_g u_C(x_0)|_g^p
\leq
\sigma H_{\partial C}(x_0)+k^p,
\]
which contradicts the strict assumption
\[
|\nabla_g u_C|_g^p
>
\sigma H_{\partial C}+k^p
\qquad\text{on }\partial C.
\]

We may therefore assume that $C\subsetneq\Omega^*$. By the comparison
principle,
\[
0\leq u_C\leq u_{\Omega^*}
\qquad\text{in }C.
\]
Since $\Omega^*$ is connected and $f\geq0$, $f\not\equiv0$, the strong
maximum principle gives
\[
u_{\Omega^*}>0\qquad\text{in }\Omega^*.
\]
Because $C\subsetneq\Omega^*$, there exists a point of
$\partial C$ belonging to $\Omega^*$. At such a point
$u_C=0$ whereas $u_{\Omega^*}>0$. Hence
\[
u_C\not\equiv u_{\Omega^*}\qquad\text{in }C.
\]
By the assumed strong comparison principle,
\[
u_C<u_{\Omega^*}\qquad\text{in }C.
\]

At the regular tangential contact point $x_0$, the two functions agree
and vanish, and the assumed interior sphere condition allows us to apply
the Hopf boundary point principle to the ordered pair
$u_{\Omega^*}\geq u_C$. With respect to the common outward normal
$\nu$, we obtain
\[
\partial_\nu u_{\Omega^*}(x_0)
<
\partial_\nu u_C(x_0)<0.
\]
Consequently,
\[
|\nabla_g u_{\Omega^*}(x_0)|_g^p
>
|\nabla_g u_C(x_0)|_g^p.
\]

On the other hand, the curvature comparison lemma gives
\[
H_{\partial\Omega^*}(x_0)
\leq
H_{\partial C}(x_0).
\]
Hence, by the strict assumption on $\partial C$,
\[
|\nabla_g u_C(x_0)|_g^p
>
\sigma H_{\partial C}(x_0)+k^p
\geq
\sigma H_{\partial\Omega^*}(x_0)+k^p.
\]
Combining the last two inequalities gives
\[
|\nabla_g u_{\Omega^*}(x_0)|_g^p
>
\sigma H_{\partial\Omega^*}(x_0)+k^p,
\]
which contradicts the contact optimality inequality.

Thus $\Gamma_0=\varnothing$.
\end{proof}

\section{Optimality condition}

\subsection{Riemannian domain deformations}

Let $X$ be a sufficiently regular vector field on $M$ and
$\Phi_t$ its local flow. Set
\[
\Omega_t=\Phi_t(\Omega).
\]
The shape derivative is
\[
d\mathcal J_{p,g}(\Omega)[X]
=
\left.\frac d{dt}
\mathcal J_{p,g}(\Omega_t)\right|_{t=0}.
\]

\subsection{Material and shape derivatives}

Let $u_t=u_{\Omega_t}$. The material derivative is defined by
transporting $u_t$ to the fixed domain:
\[
\dot u
=
\left.\frac d{dt}
(u_t\circ\Phi_t)\right|_{t=0}.
\]
The shape derivative is obtained after removing the transport
contribution:
\[
u'
=
\dot u-\langle\nabla_g u,X\rangle_g.
\]
On $\partial\Omega$, since $u=0$,
\[
u'
=
-\partial_\nu u\,g(X,\nu).
\]

As a consistency check, when $p=2$ one has $p'=2$, $\Delta_{g,p}=\Delta_g$, and the $p$-torsion functional reduces to the usual Riemannian torsion. The stress tensor and Hadamard formula below then reduce to the corresponding linear expressions.

\subsection{Shape derivative of the $p$-torsion}

We now derive the first variation of the $p$-torsional functional.
The argument is formulated intrinsically by pulling the metric back
under the deformation flow.

The assumptions in this subsection are deliberately stronger than those used for existence and stability. The latter require only $f\in L^{p'}(M)$, whereas the present boundary calculation is stated under additional regularity and shape-differentiability assumptions so that normal traces and pointwise boundary identities are meaningful.

\begin{proposition}[Hadamard formula for the Riemannian $p$-torsion]
\label{prop:hadamard-p-torsion}
Let $\Omega\in\mathcal A_C(M)$ with
$\partial\Omega\in C^{2,\alpha}$, let $1<p<\infty$, and assume
$f\in C^{1,\alpha}(M)$. Let $u=u_\Omega$ be the solution of
\[
\begin{cases}
-\Delta_{g,p}u=f & \text{in }\Omega,\\
u=0 & \text{on }\partial\Omega.
\end{cases}
\]
Assume that $u\in C^{1,\beta}(\overline\Omega)$ for some $\beta\in(0,1)$.
Assume moreover that the material state map $t\mapsto u_t\circ\Phi_t$ is differentiable at $t=0$ in a function space continuously embedded in $C^1(\overline\Omega)$, and that its derivative can be used in the weak formulation and in the transport identities below. This is the precise shape-differentiability hypothesis needed for the computation; it is not being inferred automatically from the Hölder regularity of $f$.
Let $X\in C^2(TM)$, let $\Phi_t$ be its flow, and put
\[
\Omega_t=\Phi_t(\Omega).
\]
Then
\begin{equation}
dT_{p,g}(\Omega)[X]
=
\int_{\partial\Omega}
|\partial_\nu u|^p g(X,\nu)\,dS_g.
\label{eq:hadamard-p-torsion}
\end{equation}
Equivalently,
\[
dT_{p,g}(\Omega)[X]
=
\int_{\partial\Omega}
|\nabla_g u|_g^p g(X,\nu)\,dS_g.
\]
\end{proposition}

\begin{proof}
Define
\[
E(\Omega,v)
=
\frac1p\int_\Omega|\nabla_gv|_g^p\,dV_g
-
\int_\Omega fv\,dV_g.
\]
The weak formulation of the state equation is precisely
\[
D_vE(\Omega,u)[\varphi]=0
\qquad
\text{for all }\varphi\in W_0^{1,p}(\Omega).
\]
Taking $\varphi=u$ gives
\begin{equation}
\int_\Omega|\nabla_g u|_g^p\,dV_g
=
\int_\Omega fu\,dV_g
=
T_{p,g}(\Omega).
\label{eq:torsion-energy-identity}
\end{equation}
Consequently,
\begin{equation}
E(\Omega,u)
=
-\frac1{p'}T_{p,g}(\Omega).
\label{eq:energy-torsion-relation}
\end{equation}

Let $u_t=u_{\Omega_t}$ and define the material derivative
\[
\dot u
=
\left.
\frac{d}{dt}
\right|_{t=0}
(u_t\circ\Phi_t).
\]
Since
\[
u_t=0\qquad\text{on }\partial\Omega_t,
\]
we have
\[
u_t(\Phi_t(x))=0
\qquad (x\in\partial\Omega),
\]
and hence
\[
\dot u=0
\qquad\text{on }\partial\Omega.
\]
The shape derivative
\[
u'=\dot u-g(\nabla_g u,X)
\]
therefore satisfies
\[
u'
=
-\partial_\nu u\,g(X,\nu)
\qquad\text{on }\partial\Omega.
\]

To compute the domain derivative intrinsically, set
\[
g_t=\Phi_t^*g.
\]
Then
\[
\left.\frac{d}{dt}\right|_{t=0}g_t
=
\mathcal L_Xg,
\qquad
\left.\frac{d}{dt}\right|_{t=0}dV_{g_t}
=
\operatorname{div}_gX\,dV_g.
\]
For the $p$-Dirichlet energy one obtains
\begin{align}
\left.
\frac{d}{dt}
\right|_{t=0}
\frac1p
\int_\Omega
|\nabla_{g_t}(u_t\circ\Phi_t)|_{g_t}^p\,dV_{g_t}
={}&
\int_\Omega
|\nabla_g u|_g^{p-2}
g(\nabla_g u,\nabla_g\dot u)\,dV_g
\nonumber\\
&+
\int_\Omega
\left[
\frac1p|\nabla_g u|_g^p\operatorname{div}_gX
-
|\nabla_g u|_g^{p-2}
g(\nabla_{\nabla_g u}X,\nabla_g u)
\right]dV_g.
\label{eq:transport-p-energy}
\end{align}
The expression in the second line is
\[
S_u:\nabla X,
\]
where
\begin{equation}
S_u
=
\frac1p|\nabla_g u|_g^p g
-
|\nabla_g u|_g^{p-2}du\otimes du.
\label{eq:stress-tensor-p}
\end{equation}
Here $S_u:\nabla X$ denotes the full contraction
\[
S_u:\nabla X
=
\sum_{i,j=1}^n
S_u(e_i,e_j)\,
g(\nabla_{e_i}X,e_j),
\]
where $(e_1,\ldots,e_n)$ is any local $g$-orthonormal frame.
Equivalently,
\[
S_u:\nabla X
=
\frac1p|\nabla_g u|_g^p\operatorname{div}_g X
-
|\nabla_g u|_g^{p-2}
g(\nabla_{\nabla_g u}X,\nabla_g u).
\]
The source term has derivative
\[
\left.
\frac{d}{dt}
\right|_{t=0}
\int_\Omega
(f\circ\Phi_t)(u_t\circ\Phi_t)\,dV_{g_t}
=
\int_\Omega
\left[
f\dot u
+
u\,g(\nabla_gf,X)
+
fu\,\operatorname{div}_gX
\right]dV_g.
\]
The terms involving $\dot u$ cancel by the weak state equation.
Thus
\begin{align}
dE(\Omega,u)[X]
={}&
\int_\Omega S_u:\nabla X\,dV_g
-
\int_\Omega
u\,g(\nabla_gf,X)\,dV_g
-
\int_\Omega
fu\,\operatorname{div}_gX\,dV_g.
\label{eq:energy-shape-reduced}
\end{align}

We next use the divergence identity for $S_u$. Since
\[
-\operatorname{div}_g
\left(
|\nabla_g u|_g^{p-2}\nabla_g u
\right)
=f,
\]
a direct covariant differentiation gives
\begin{equation}
\operatorname{div}_g S_u
=
f\nabla_g u.
\label{eq:stress-divergence}
\end{equation}
The identity \eqref{eq:stress-divergence} is understood in the distributional sense under the weak formulation. In a smooth regime it follows by covariant differentiation in a local $g$-orthonormal frame, using
$\nabla(\frac1p|\nabla_g u|_g^p)=\nabla^2u(\nabla_g u,\cdot)$,
the covariant product rule, and
$-\operatorname{div}_g(|\nabla_g u|_g^{p-2}\nabla_g u)=f$.
Thus the computation is intrinsic; no Euclidean divergence identity is being assumed.
Moreover,
\[
-u\,g(\nabla_gf,X)-fu\,\operatorname{div}_gX
=
-\operatorname{div}_g(fuX)
+
f\,g(\nabla_g u,X).
\]
Hence, by the divergence theorem and $u=0$ on $\partial\Omega$,
\begin{align}
dE(\Omega,u)[X]
={}&
\int_{\partial\Omega}S_u(X,\nu)\,dS_g
-
\int_\Omega f\,g(\nabla_g u,X)\,dV_g
+
\int_\Omega f\,g(\nabla_g u,X)\,dV_g
\nonumber\\
={}&
\int_{\partial\Omega}S_u(X,\nu)\,dS_g.
\label{eq:energy-boundary-representation}
\end{align}

On $\partial\Omega$ we have
\[
\nabla_g u=(\partial_\nu u)\nu.
\]
Therefore
\[
S_u(X,\nu)
=
\left(\frac1p-1\right)
|\partial_\nu u|^p g(X,\nu)
=
-\frac1{p'}
|\partial_\nu u|^p g(X,\nu).
\]
Consequently,
\begin{equation}
dE(\Omega,u)[X]
=
-\frac1{p'}
\int_{\partial\Omega}
|\partial_\nu u|^p g(X,\nu)\,dS_g.
\label{eq:energy-hadamard-p}
\end{equation}
Finally, differentiating \eqref{eq:energy-torsion-relation} gives
\[
dT_{p,g}(\Omega)[X]
=
-p'\,dE(\Omega,u)[X],
\]
and hence
\[
dT_{p,g}(\Omega)[X]
=
\int_{\partial\Omega}
|\partial_\nu u|^p g(X,\nu)\,dS_g.
\]
Since
\[
|\partial_\nu u|
=
|\nabla_g u|_g
\qquad\text{on }\partial\Omega,
\]
formula \eqref{eq:hadamard-p-torsion} follows.
\end{proof}

\subsection{First variation of the geometric terms}

The first variation formulas for volume and perimeter give
\[
dV_g(\Omega)[X]
=
\int_{\partial\Omega}g(X,\nu)\,dS_g,
\]
and
\[
dP_g(\Omega)[X]
=
\int_{\partial\Omega}
H_{\partial\Omega}g(X,\nu)\,dS_g.
\]

\subsection{The Riemannian shape derivative of $\mathcal J_{p,g}$}

Combining the Hadamard formula \eqref{eq:hadamard-p-torsion} with the first variation formulas for perimeter and volume,
\[
\boxed{
d\mathcal J_{p,g}(\Omega)[X]
=
\int_{\partial\Omega}
\left(
\sigma H_{\partial\Omega}
+k^p
-
|\nabla_g u_\Omega|_g^p
\right)
g(X,\nu)\,dS_g.
}
\]
Set
\[
F_\Omega
=
\sigma H_{\partial\Omega}
+k^p
-
|\nabla_g u_\Omega|_g^p.
\]

\subsection{First-order optimality}

If $\Omega^\ast$ minimizes $\mathcal J_{p,g}$, then for every
one-sided admissible deformation (that is, for sufficiently small
$t\geq0$ such that $\Omega_t$ remains in the admissible class),
\[
d\mathcal J_{p,g}(\Omega^\ast)[X]\geq0.
\]
Therefore
\[
\int_{\partial\Omega^\ast}
F_{\Omega^\ast}\,
g(X,\nu)\,dS_g
\geq0.
\]

\subsection{The free boundary}

Let $x_0\in\Gamma$ be a regular free-boundary point. Localized
deformations supported near $x_0$ can be chosen with either sign of
the normal component, with both $t>0$ and $t<0$ remaining admissible.
Hence
\[
\int_\Gamma
F_{\Omega^\ast}\varphi\,dS_g=0
\]
for every $\varphi\in C_c^\infty(\Gamma)$, and the fundamental lemma
of the calculus of variations gives
\[
F_{\Omega^\ast}=0
\qquad\text{on }\Gamma.
\]
This is exactly the free-boundary condition.

\subsection{The contact boundary}

At a regular contact point $x_0\in\Gamma_0$, the constraint
$C\subset\Omega^\ast$ imposes a one-sided condition on the admissible
normal velocity. With the outward-normal convention used here, an
admissible localized deformation satisfies
\[
g(X,\nu_{\Omega^\ast})\geq0
\qquad\text{on the contact set.}
\]
To localize the first variation, let $U$ be a sufficiently small
neighborhood of $x_0$ in $\partial\Omega^\ast$. For every
$\varphi\in C_c^\infty(U)$ with $\varphi\geq0$, the extension of $X$
can be chosen so that
$g(X,\nu_{\Omega^\ast})=\varphi$ on $U\cap\partial\Omega^\ast$ and
its local flow preserves $C\subset\Omega_t$ for all sufficiently small
$t\geq0$. Since $\Omega^\ast$ is minimizing, the one-sided first variation
therefore gives
\[
\int_{U\cap\partial\Omega^\ast}
F_{\Omega^\ast}\varphi\,dS_g\geq0.
\]
On portions of $U\cap\partial\Omega^\ast$ which are free rather than in
contact, the first-order free-boundary condition already gives
$F_{\Omega^\ast}=0$. Hence
\[
\int_{U\cap\Gamma_0}F_{\Omega^\ast}\varphi\,dS_g\geq0.
\]
Since $x_0$ is arbitrary and $F_{\Omega^\ast}$ is continuous under the
regularity assumptions of the shape derivative, this yields
\[
F_{\Omega^\ast}(x_0)\geq0.
\]
Thus
\[
|\nabla_g u_{\Omega^\ast}|_g^p
\leq
\sigma H_{\partial\Omega^\ast}+k^p
\]
at every regular contact point. This is the contact optimality inequality.

\subsection{Optimality condition at tangential contact}

At a regular contact point $x_0$,
\[
\nu_{\Omega^\ast}(x_0)=\nu_C(x_0),
\]
and Lemma~\ref{lem:curvature-comparison} gives
\[
H_{\partial\Omega^\ast}(x_0)
\leq H_{\partial C}(x_0).
\]
Thus the contact optimality condition implies
\[
|\nabla_g u_{\Omega^\ast}(x_0)|_g^p
\leq
\sigma H_{\partial C}(x_0)+k^p.
\]
This is the inequality needed in the contradiction argument for the
sufficient condition.

\subsection{The resulting variational inequality}

The first-order optimality system is
\[
\boxed{
\begin{cases}
|\nabla_g u_{\Omega^\ast}|_g^p
=
\sigma H_{\partial\Omega^\ast}+k^p,
&\text{on }\partial\Omega^\ast\setminus\partial C,\\[1mm]
|\nabla_g u_{\Omega^\ast}|_g^p
\leq
\sigma H_{\partial\Omega^\ast}+k^p,
&\text{on }\partial\Omega^\ast\cap\partial C.
\end{cases}}
\]

\section{Proofs of the main theorems}

\subsection{Proof of the existence theorem}

Let
\[
m=\inf_{\Omega\in \mathcal A_C(M)}\mathcal J_{p,g}(\Omega).
\]
For $\Omega\in \mathcal A_C(M)$, the uniform energy estimate established
above gives
\[
T_{p,g}(\Omega)
\leq
C\|f\|_{L^{p'}(M)}^{p'}.
\]
Since the perimeter and volume terms are non-negative, we obtain
\[
\mathcal J_{p,g}(\Omega)
=
-T_{p,g}(\Omega)+\sigma P_g(\Omega)+k^pV_g(\Omega)
\geq
-C\|f\|_{L^{p'}(M)}^{p'}.
\]
Hence
\[
m=\inf_{\Omega\in\mathcal A_C(M)}\mathcal J_{p,g}(\Omega)>-\infty.
\]

Choose a minimizing sequence $(\Omega_j)\subset \mathcal A_C(M)$. By the compactness
theorem, after extraction,
\[
\Omega_j\to\Omega^\ast\in \mathcal A_C(M).
\]
By stability,
\[
u_{\Omega_j}\to u_{\Omega^\ast}
\quad\text{strongly in }W^{1,p}(M),
\]
and therefore
\[
T_{p,g}(\Omega_j)\to T_{p,g}(\Omega^\ast).
\]
Moreover,
\[
V_g(\Omega_j)\to V_g(\Omega^\ast)
\]
and the Riemannian perimeter is lower semicontinuous:
\[
P_g(\Omega^\ast)
\leq
\liminf_{j\to\infty}P_g(\Omega_j).
\]
Since the torsional and volume terms converge, we have
\[
\begin{aligned}
\mathcal J_{p,g}(\Omega^\ast)
&=-T_{p,g}(\Omega^\ast)+\sigma P_g(\Omega^\ast)
   +k^pV_g(\Omega^\ast)\\
&\leq -\lim_{j\to\infty}T_{p,g}(\Omega_j)
   +\sigma\liminf_{j\to\infty}P_g(\Omega_j)
   +k^p\lim_{j\to\infty}V_g(\Omega_j)\\
&\leq \liminf_{j\to\infty}
   \mathcal J_{p,g}(\Omega_j)
=m.
\end{aligned}
\]
Since $m$ is the infimum, necessarily
\[
\mathcal J_{p,g}(\Omega^\ast)=m,
\]
and $\Omega^\ast$ is a minimizer.

\subsection{Proof of the comparison result}

Let
\[
w=(u_C-u_\Omega)^+.
\]
Since $C\subset\Omega$ and $u_\Omega\geq0$, we have
$w\in W_0^{1,p}(C)$. Testing the difference of the two equations by
$w$ gives
\[
\int_C
\left(
|\nabla_g u_C|_g^{p-2}\nabla_g u_C
-
|\nabla_g u_\Omega|_g^{p-2}\nabla_g u_\Omega
\right)
\cdot\nabla_gw\,dV_g=0.
\]
Strict monotonicity implies $\nabla_gw=0$, hence $w=0$. Therefore
\[
u_C\leq u_\Omega\quad\text{in }C.
\]
The non-negativity follows from the weak maximum principle.

\subsection{Proof of the contact optimality condition}

Let
\[
x_0\in\Gamma_0
=
\partial\Omega^\ast\cap\partial C
\]
be a regular contact point. Since
\[
C\subset\Omega^\ast,
\]
the minimizer $\Omega^\ast$ cannot be deformed arbitrarily in both
normal directions at $x_0$. Indeed, an admissible deformation must
preserve the inclusion constraint
\[
C\subset\Omega_t.
\]

Let $X$ be a $C^2$ vector field on $M$, let $\Phi_t$ denote its local
flow, and define
\[
\Omega_t=\Phi_t(\Omega^\ast).
\]
We restrict ourselves to vector fields for which
\[
\Omega_t\in\mathcal A_C(M)
\qquad\text{for all }0\leq t<t_0.
\]
Since $\Omega^\ast$ is a minimizer of $\mathcal J_{p,g}$ over
$\mathcal A_C(M)$, we have
\[
\mathcal J_{p,g}(\Omega_t)
\geq
\mathcal J_{p,g}(\Omega^\ast)
\qquad\text{for }0\leq t<t_0.
\]
Therefore,
\[
\left.
\frac{d}{dt}
\mathcal J_{p,g}(\Omega_t)
\right|_{t=0^+}
\geq0,
\]
that is,
\begin{equation}
d\mathcal J_{p,g}(\Omega^\ast)[X]\geq0.
\label{eq:one-sided-contact-variation}
\end{equation}

We now determine the geometric restriction imposed on the normal
component of $X$ at the contact boundary.

At a regular contact point $x_0$, the inclusion
\[
C\subset\Omega^\ast
\]
implies that $\partial C$ and $\partial\Omega^\ast$ have a common
supporting tangent hyperplane. Hence
\[
T_{x_0}\partial C
=
T_{x_0}\partial\Omega^\ast
\]
and, with the outward-normal convention,
\[
\nu_C(x_0)
=
\nu_{\Omega^\ast}(x_0).
\]

For an admissible deformation preserving
\[
C\subset\Omega_t,
\]
the boundary of $\Omega^\ast$ may move away from $C$, but it cannot
move through $C$. Consequently, its normal velocity at a contact point
must satisfy
\begin{equation}
g(X,\nu_{\Omega^\ast})(x_0)\geq0.
\label{eq:contact-normal-velocity}
\end{equation}

This is the fundamental difference between a free-boundary point and a
contact point. At a free-boundary point, both signs of the normal
velocity are admissible, whereas at a contact point only the
outward-pointing direction is admissible.

We next use the first variation formula established previously. For
every sufficiently regular vector field $X$,
\[
d\mathcal J_{p,g}(\Omega^\ast)[X]
=
\int_{\partial\Omega^\ast}
\left(
\sigma H_{\partial\Omega^\ast}
+k^p
-
|\nabla_g u_{\Omega^\ast}|_g^p
\right)
g(X,\nu_{\Omega^\ast})\,dS_g.
\]

Define
\[
F_{\Omega^\ast}
=
\sigma H_{\partial\Omega^\ast}
+k^p
-
|\nabla_g u_{\Omega^\ast}|_g^p.
\]
Then \eqref{eq:one-sided-contact-variation} becomes
\begin{equation}
\int_{\partial\Omega^\ast}
F_{\Omega^\ast}
g(X,\nu_{\Omega^\ast})\,dS_g
\geq0.
\label{eq:global-contact-variation}
\end{equation}

We now localize the argument near $x_0$.

Let
\[
U\subset\partial\Omega^\ast
\]
be a sufficiently small relatively open neighborhood of $x_0$. Choose
\[
\varphi\in C_c^\infty(U),
\qquad
\varphi\geq0,
\qquad
\varphi(x_0)>0.
\]

Because $x_0$ is a regular contact point, one can construct a
sufficiently regular vector field $X$, supported in a small
neighborhood of $U$, such that
\[
g(X,\nu_{\Omega^\ast})=\varphi
\qquad\text{on }U
\]
and
\[
g(X,\nu_{\Omega^\ast})=0
\qquad\text{on }\partial\Omega^\ast\setminus U.
\]
The localization is taken in a neighborhood $U\subset\partial\Omega^\ast$
of the contact point, not necessarily in a neighborhood contained in the
contact set. Choose $\varphi\in C_c^\infty(U)$ with $\varphi\geq0$ and
$\varphi(x_0)>0$. The extension of $X$ can be chosen so that
$g(X,\nu_{\Omega^\ast})=\varphi$ on $U$ and vanishes on
$\partial\Omega^\ast\setminus U$. Since the normal velocity is
non-negative at every contact point in the support of $\varphi$, and is
strictly positive at $x_0$, the resulting local flow preserves
$C\subset\Omega_t$ for all sufficiently small $t\geq0$. Points of the
contact set where $\varphi=0$ are left fixed, while points where
$\varphi>0$ are displaced outward from $C$; away from the contact set the
inclusion is preserved for small time by the positive separation from
$C$. This gives the required one-sided admissible variation.

Substituting this vector field into
\eqref{eq:global-contact-variation}, we obtain
\[
\int_U F_{\Omega^\ast}\varphi\,dS_g\geq0.
\]
If $F_{\Omega^\ast}(x_0)<0$, continuity of $F_{\Omega^\ast}$ gives a
smaller neighborhood on which $F_{\Omega^\ast}<0$. Choosing a
non-negative $\varphi\in C_c^\infty(U)$ with $\varphi(x_0)>0$ and support
inside that neighborhood yields a contradiction. Hence
\[
F_{\Omega^\ast}(x_0)\geq0.
\]

Since $x_0$ was arbitrary, we conclude that
\[
F_{\Omega^\ast}\geq0
\qquad\text{on }\Gamma_0.
\]
Therefore,
\[
\sigma H_{\partial\Omega^\ast}
+k^p
-
|\nabla_g u_{\Omega^\ast}|_g^p
\geq0
\qquad\text{on }\Gamma_0,
\]
or equivalently,
\[
\boxed{
|\nabla_g u_{\Omega^\ast}|_g^p
\leq
\sigma H_{\partial\Omega^\ast}
+k^p
\qquad\text{on }\Gamma_0.
}
\]

which proves the contact optimality inequality stated above.
\section{Examples}
\begin{example}[Radial solutions on geodesic balls]
Let
\[
\Omega=B_g(o,\rho)
\]
be a geodesic ball of radius $\rho$, with $\rho$ smaller than the injectivity
radius. Its boundary is a smooth geodesic sphere.

For this radial calculation, assume in addition that the metric is rotationally
symmetric about $o$, so that on $B_g(o,\rho)$
\[
g=dr^2+G(r)^2g_{\mathbb S^{n-1}},
\qquad J(r)=G(r)^{n-1}.
\]
We also assume the usual smoothness conditions at the center. If the source is
radial, $f(x)=F(r)$, then the state is radial:

\[
u(x)=U(r),
\qquad
r=d_g(o,x).
\]
The $p$-Laplace equation becomes
\[
-\frac1{J(r)}
\frac d{dr}
\left(
J(r)|U'(r)|^{p-2}U'(r)
\right)
=
F(r),
\]
where $J(r)$ is the radial volume density. For the rotationally symmetric (warped) metric above, the outward mean curvature of the geodesic sphere $r=\rho$ is

\[
H_{\partial B_g(o,\rho)}=(n-1)\frac{G'(\rho)}{G(\rho)}.
\]

If $U'(r)\leq0$, integration
from $0$ to $r$ gives
\[
J(r)|U'(r)|^{p-2}(-U'(r))
=
\int_0^rJ(s)F(s)\,ds.
\]
Consequently,
\[
|U'(r)|^p
=
\left[
\frac{\displaystyle\int_0^rJ(s)F(s)\,ds}
{J(r)}
\right]^{p'}.
\]
Thus the sufficient condition on $\partial B_g(o,\rho)$ becomes
\[
\left[
\frac{\displaystyle\int_0^\rho J(s)F(s)\,ds}
{J(\rho)}
\right]^{p'}
>
\sigma H_{\partial B_g(o,\rho)}+k^p.
\]

For completeness, under the rotational symmetry assumption above,
\[
dV_g=J(r)\,dr\,dS_{\mathbb S^{n-1}},
\qquad
\nabla_gu=U'(r)\partial_r,
\qquad
|\nabla_gu|_g=|U'(r)|.
\]
For a radial vector field $Y=a(r)\partial_r$,
\[
\operatorname{div}_gY
=
\frac1{J(r)}\frac{d}{dr}\bigl(J(r)a(r)\bigr),
\]
which gives the displayed one-dimensional equation. Regularity at the
center yields
\[
J(r)|U'(r)|^{p-2}U'(r)\longrightarrow0
\qquad\text{as }r\downarrow0.
\]
Since $F\geq0$, the solution is non-increasing, and hence
\[
|U'(r)|
=
\left[
\frac{\displaystyle\int_0^rJ(s)F(s)\,ds}{J(r)}
\right]^{1/(p-1)}.
\]
Raising to the power $p$ gives the exponent
\[
p'=\frac{p}{p-1}.
\]
This example therefore reduces the sufficient condition to a one-dimensional
flux balance between the accumulated source and the radial volume density.
\end{example}

\begin{example}[The round sphere with radial data]
Consider the round sphere
\[
S_R^n=\{x\in\mathbb R^{n+1}:|x|=R\}
\]
with its round metric. In geodesic polar coordinates around a pole,
\[
g=dr^2+R^2\sin^2(r/R)g_{\mathbb S^{n-1}},
\]
and therefore
\[
J(r)=\left(R\sin\frac rR\right)^{n-1}.
\]
For
\[
C=B_{S_R^n}(o,\rho),
\qquad
0<\rho<\frac{\pi R}{2},
\]
the outward mean curvature is
\[
H_{\partial C}
=
\frac{n-1}{R}\cot\left(\frac\rho R\right).
\]
If $f(x)=F(r)$ with $F\geq0$, then the boundary flux is
\[
|\nabla_g u_C|_g^p\big|_{\partial C}
=
\left[
\frac{
\displaystyle
\int_0^\rho
\left(R\sin\frac sR\right)^{n-1}F(s)\,ds
}
{\left(R\sin\frac\rho R\right)^{n-1}}
\right]^{p'}.
\]
Consequently, under the additional regularity and comparison hypotheses
of Theorem~\ref{thm:sufficient-condition}, that theorem yields
\[
\boxed{
\left[
\frac{
\displaystyle
\int_0^\rho
\left(R\sin\frac sR\right)^{n-1}F(s)\,ds
}
{\left(R\sin\frac\rho R\right)^{n-1}}
\right]^{p'}
>
\frac{\sigma(n-1)}{R}\cot\left(\frac\rho R\right)+k^p.
}
\]

In particular, for the explicit model calculation with a constant radial
source $F(r)=F_0>0$,
this source is not compatible with the standing assumption
$\operatorname{supp}f\subset K\Subset C$ if it is taken on all of $C$;
thus the following formulas are illustrative and are not, by themselves,
an application of Theorem~\ref{thm:sufficient-condition}.
\[
\boxed{
\left[
F_0
\frac{
\displaystyle
\int_0^\rho
\left(R\sin\frac sR\right)^{n-1}\,ds
}
{\left(R\sin\frac\rho R\right)^{n-1}}
\right]^{p'}
>
\frac{\sigma(n-1)}{R}\cot\left(\frac\rho R\right)+k^p.
}
\]
In dimension $n=2$ this becomes completely explicit (for this constant-source radial model):
\[
\int_0^\rho R\sin\left(\frac sR\right)F_0\,ds
=
F_0R^2\left(1-\cos\frac\rho R\right),
\]
and hence
\begin{equation}
\boxed{
\left[
F_0R\tan\left(\frac\rho{2R}\right)
\right]^{p'}
>
\frac{\sigma}{R}\cot\left(\frac\rho R\right)+k^p.
}
\label{eq:spherical-condition-2d}
\end{equation}

For $p=2$, $p'=2$, and the criterion becomes
\[
\left[F_0R\tan\left(\frac\rho{2R}\right)\right]^2
>
\frac{\sigma}{R}\cot\left(\frac\rho R\right)+k^2,
\]
which is the corresponding quadratic torsion structure. The restriction
$0<\rho<\pi R/2$ ensures
\[
H_{\partial C}>0,
\]
so the reference geodesic sphere is strictly mean-convex with respect to
its outward normal. For small $\rho$, the curvature behaves like $1/\rho$.
\end{example}

\begin{example}[Euclidean large-radius limit]
The spherical criterion continuously recovers the Euclidean criterion when
$R\to\infty$ with the geodesic radius $\rho$ fixed. In dimension $n=2$,
\eqref{eq:spherical-condition-2d} reads
\[
\left[
F_0R\tan\left(\frac\rho{2R}\right)
\right]^{p'}
>
\frac\sigma R\cot\left(\frac\rho R\right)+k^p.
\]
Using
\[
\tan t\sim t,
\qquad
\cot t\sim\frac1t
\qquad(t\to0),
\]
we have
\[
R\tan\left(\frac\rho{2R}\right)\longrightarrow\frac\rho2,
\qquad
\frac1R\cot\left(\frac\rho R\right)\longrightarrow\frac1\rho.
\]
Therefore
\[
\boxed{
\left(\frac{F_0\rho}{2}\right)^{p'}
>
\frac\sigma\rho+k^p.
}
\]
This is the corresponding Euclidean model criterion for a disk of radius
$\rho$ with constant source $F_0$. As above, the constant-source model
is illustrative rather than an instance of the standing compact-support
hypothesis. The limit also confirms
that the nonlinear contribution retains the conjugate exponent
$p'=p/(p-1)$, while the spherical mean curvature converges to the
Euclidean curvature $1/\rho$.
\end{example}

\end{document}